%% file: allocation_TAC.tex
\documentclass[10pt,twocolumn,twoside]{IEEEtran}

\makeatletter
\long\def\@makecaption#1#2{\ifx\@captype\@IEEEtablestring%
\footnotesize\begin{center}{\normalfont\footnotesize #1}\\
{\normalfont\footnotesize\scshape #2}\end{center}%
\@IEEEtablecaptionsepspace
\else
\@IEEEfigurecaptionsepspace
\setbox\@tempboxa\hbox{\normalfont\footnotesize {#1.}~~ #2}%
\ifdim \wd\@tempboxa >\hsize%
\setbox\@tempboxa\hbox{\normalfont\footnotesize {#1.}~~ }%
\parbox[t]{\hsize}{\normalfont\footnotesize \noindent\unhbox\@tempboxa#2}%
\else
\hbox to\hsize{\normalfont\footnotesize\hfil\box\@tempboxa\hfil}\fi\fi}
\makeatother

\usepackage{amsmath, amsthm, amsfonts, amssymb, graphicx, cite, epsfig,epstopdf,color,soul,mathtools}
\usepackage{mathrsfs, mathtools, bbm, dsfont}

\usepackage{subcaption}
\usepackage{ifthen}
\usepackage{multirow, multicol}
\usepackage{float}
\usepackage{subeqnarray, array}
\usepackage{xcolor}
\usepackage[font=small]{caption}
\usepackage{tikz}

\usepackage{enumerate}
\usepackage[shortlabels]{enumitem}
\usepackage{empheq}
\usepackage{multicol}

\usepackage[font=normalsize]{caption}

\usepackage[bookmarks=true,pageanchor,colorlinks,linkcolor=red,anchorcolor=blue, citecolor=blue,urlcolor=blue,hyperfootnotes=false]{hyperref}
		
\usepackage{algorithm}
\usepackage{algpseudocode}

\algnewcommand{\Initialize}[1]{%
  \State \textbf{Initialize:}
  \Statex \hspace*{\algorithmicindent}\parbox[t]{.8\linewidth}{\raggedright #1}
}
\algdef{SE}[DOWHILE]{Do}{doWhile}{\algorithmicdo}[1]{\algorithmicwhile\ #1}%

\input{header}
\title{\LARGE \bf
Distributed Resource Allocation Over Dynamic Networks with Uncertainty
}

\author{Thinh T. Doan\thanks{Thinh T. Doan is with the H. Milton Stewart School of
Industrial and Systems Engineering, the Georgia Institute of Technology,  Atlanta, GA, USA
        {\tt\small ttdoan2@illinois.edu}} and Carolyn L. Beck\thanks{Carolyn L. Beck is with the Department of Industrial and Enterprise Systems Engineering, University of Illinois,  Urbana, IL, USA
        {\tt\small beck3@illinois.edu}}}
\date{}
\begin{document}
\maketitle

\begin{abstract}
Motivated by broad applications in various fields of engineering, we study a network resource allocation problem where the goal is to optimally allocate a fixed quantity of resources over a network of nodes. We consider large scale networks with complex interconnection structures, thus any solution must be implemented in parallel and based only on local data resulting in a need for distributed algorithms. In this paper, we study a distributed Lagrangian method for such problems. By utilizing the so-called distributed subgradient methods to solve the dual problem, our approach eliminates the need for central coordination in updating the dual variables, which is often required in classic Lagrangian methods. Our focus is to understand the performance of this distributed algorithm when the number of resources is unknown and may be time-varying. In particular, we obtain an upper bound on the convergence rate of the algorithm to the optimal value, in expectation, as a function of the topology of the underlying network. The effectiveness of the proposed method is demonstrated by its application to the economic dispatch problem in power systems, with simulations completed  on the benchmark IEEE-14 and IEEE-118 bus test systems.
\end{abstract}

\input{intro.tex}

\input{alg_DLM.tex}

\input{alg_DRLM.tex}

\input{simulation.tex}

\section{Concluding Remarks}\label{sec:Conclusion}
In this paper, we first introduce our Distributed Lagrangian Method, a fully distributed version of well-known Lagrangian methods, for resource allocation problems over time-varying networks. In particular, this method allows for distributed subgradient algorithms on local copies of the Lagrange multiplier. We then study a Distributed Stochastic Lagrangian Method for the problem when the number of resources is unknown precisely. We show that our algorithm is robust to such uncertainty. Specifically, our algorithm converges at a rate $\mathcal{O}(n\ln(k)/\delta\sqrt{k})$ in expectation to the optimal value when the step-size decays as $\alpha(k) = 1/\sqrt{k}$. Finally, we illustrate the efficacy of our methods for solving distributed economic dispatch problems on IEEE-14 and IEEE-118 bus test systems.

Although it may appear initially that our proposed method is limited to the study of simple resource allocation problems over networks, we believe this approach will provide a useful tool for more general problems in many areas, especially in power systems, which is our main motivation. Potential applications of our method following this work include additional control problems in power systems, for example, frequency control problems\cite{StevenLow2014, Dorfler2016}, where distributed controllers are preferred to decentralized controllers; and multi-area optimal dispatch problems between connected regional power systems \cite{BaldickChatterjee14}. We believe that our proposed framework can be generalized to solve these problems, which we leave for future studies.

\section*{Acknowledgements}
The first author would like to thank Alex Olshevsky for many useful discussions. This work was partially funded by NSF grants ECCS 15-09302 and CNS 15-44953.

\bibliographystyle{plain}
\bibliography{refs}

\appendix
\input{proofs.tex}

\end{document}

%% file: header.tex
\newcommand{\Eset}{\mathbb{E}}

\newcommand{\Rset}{\mathbb{R}}

\newcommand{\Fcal}{{\cal F}}
\newcommand{\Gcal}{{\cal G}}

\newcommand{\Lcal}{{\cal L}}

\newcommand{\Vcal}{{\cal V}}

\newcommand{\Xcal}{{\cal X}}

\newcommand{\Psf}{\sf{P}}

\newcommand{\Abf}{{\bf A}}

\newcommand{\Ibf}{{\bf I}}

\newcommand{\vbf}{{\bf v}}

\newcommand{\xbf}{{\bf x}}

\newcommand{\1}{{\mathbf{1}}}
\newcommand{\bld}{{\boldsymbol{\lambda}}}

\renewcommand{\v}[1]{{\mathbf{#1}}}

\newtheorem{thm}{Theorem}
\newtheorem{lem}{Lemma}

\newtheorem{cor}{Corollary}
\newtheorem{remark}{Remark}
\newtheorem{notation}{\bf{Notation}}
\newtheorem{assump}{Assumption}

%% file: intro.tex

\section{Introduction}\label{sec:intro}
Motivated by numerous applications in engineering, we consider an optimization problem, defined over a network of $n$ nodes, of the form	

\begin{subequations}  
\label{prob:main}
\begin{empheq}[left={\sf P} \ : \ \empheqlbrace]{alignat=5}
& {\underset{x_1,x_2,\ldots,x_n}{\text{minimize}}}
	 &&  \sum_{i=1}^n f_i(x_i), \notag\\
 & \text{subject to}   \quad
	&& x_i \in \Xcal_i, \\
	&&& \sum_{i=1}^n (x_i - b_i) = 0, \label{prob:couple_const} 
	\end{empheq}
\end{subequations}
where $f_i:\Rset\rightarrow\Rset$ is a proper convex function and $\Xcal_i\subset\Rset$ is a compact convex set, which are known only by node $i$. Here $b_i$ is some constant, which is initially assigned at node $i$. We assume that each node $i$ is an agent with computational capabilities that can communicate with other agents (referred to as node $i$'s neighbors) connected via a given graph $\Gcal$, with interconnection structure defined by an adjacency matrix, $\Abf(k)$. Here $k$ is a time index, i.e., the graph may be either fixed or time-varying. We are interested in distributed algorithms for solving problem ${\Psf}$, meaning that each node is only allowed to send/exchange messages with its neighbors.    

Problem {\Psf} is often referred to as a network resource allocation problem, where the goal is to optimally allocate a fixed quantity of resource $\sum_{i=1}^n b_i$ over a network of nodes. Each node $i$ suffers a cost given by function $f_i$ of the amount of resource $x_i$ allocated to it. The goal of this problem is to seek an optimal allocation such that the total cost $\sum_{i=1}^n f_i(x_i)$ incurred over the network is minimized while satisfying the nodes' local constraints, i.e., $x_i\in\Xcal_i$. Often problem ${\Psf}$ is described in terms of utility functions, where each function is the nodes' utility and the goal is to maximize the total utility. This problem is traditionally solved by a central coordinator that can observe the utilities or loss functions of all nodes. Such a central coordinator, however, is undesirable and frequently unavailable for two major reasons: (1) the network is too large with too complex an interconnection structure; and (2) the nodes are
geographically distributed and have heterogeneous objectives. These reasons necessitate distributed solution architectures that will allow one to bypass the use of a central coordinator. Our focus, therefore, is to study distributed algorithms for solving problem ${\Psf}$, which can be implemented in parallel and do not require any central coordination.

Network resource allocation is a fundamental and important problem that arises in a variety of application domains within engineering. One standard example is the problem of congestion control where the global objective is to route and schedule information in a large-scale internet network such that a fair resource allocation between users is achieved \cite{Srikant2004}. Another example is coverage control problems in wireless sensor networks, where the goal is to optimally allocate a large number of sensors to an unknown environment such that the coverage area is maximized \cite{Cortes2004,SharmaSB2012}. Furthermore, resource allocation may be viewed as a simplification of the important economic dispatch problem in power systems, wherein geographically distributed generators of electricity must coordinate to meet a fixed demand while maintaining the stability of the system \cite{StevenLow2014,Dorfler2015,DoanBB2017}.

\subsection{Related work}
The study of optimization problems of the form of problem  $\Psf$ has a long history and has received much interest. A decentralized approach to solve for such problems via the so-called \textit{Lagrangian method} can be found in standard texts; for example, see \cite{Bertsekas1999, SYbook2014, Srikant2004}.  In this approach, one constructs a Lagrangian function for problem {\Psf} and sequentially updates the primal and dual variables. Due to the structure of {\Psf} the primal variables can be updated in a decentralized fashion; however, a central coordinator is required to update and distribute the dual variables to the nodes, making this approach not fully distributed.

On the other hand, the first algorithm which could be implemented in a distributed manner was the ``center-free'' method studied in \cite{HoServiSuri80} where the authors consider a relaxation of {\Psf}, that is, the nodes' local constraints are not considered. The term ``center-free'' was originally meant to refer to the absence of any central coordinator. The work in  \cite{HoServiSuri80} has lead to a number of subsequent studies \cite{XiaoBoyd06, LakshmananDeFarias08,  Necoara13, Doan17}, with the main focus on analyzing the performance of the algorithm and its variants for this relaxed problem. The primary idea of these algorithms is based on necessary and sufficient optimality conditions, namely a consensus condition on the nodes' derivatives, and a feasibility condition on the total number of resources \cite{XiaoBoyd06}.

Although distributed solutions of the relaxed problem are well-studied, their applications are limited due to the unrealistic assumption that there are no local constraints on the nodes. For example, in economic dispatch problems these local constraints, which represent the limited capacity of generators, are inevitable. Motivated by the necessary and sufficient optimality conditions of the relaxed problems in \cite{XiaoBoyd06}, there are a number of recent results on distributed methods for problem $\Psf$. In particular, the authors in \cite{DominguezGarcia2012, Yang2013, Binetti2014, Kar2012, Xing2015} use these conditions to study economic dispatch problems where objective functions are assumed to be quadratic. The authors in  \cite{CortesChekuri2015, LakshmananDeFarias08} relax the assumption on quadratic costs to convex cost functions with Lipschitz continuous gradients, and  consider relaxed problems by using appropriate penalty functions for the nodes' local constraints. In a similar approach, the authors in \cite{NedicOS2017} consider problem \eqref{prob:main} with general non-smooth convex cost functions and propose a method with a convergence rate $o(1/k)$ where $k$ is the number of iterations.       

Recently, in simultaneous work presented in \cite{DoanB2016} and \cite{Yang2016}, a distributed Lagrangian method is proposed for problem ${\Psf}$. The hallmark of this approach is the elimination of the need for a central coordinator to update the dual variables, where these authors employ the distributed subgradient method presented in \cite{Nedic2009} to solve the dual problem of ${\Psf}$. In particular, in \cite{DoanB2016}, a distributed Lagrangian algorithm for problem ${\Psf}$ on an undirected and static network is proposed. The authors in \cite{Yang2016} consider the same approach for the case of time-varying directed networks. However, the work in \cite{Yang2016} requires an assumption of strict convexity of the objective functions, whereas in \cite{DoanB2016} it is assumed only that the functions are convex. We also note some related work is presented in \cite{Aybat1, Aybat2}, in which distributed primal-dual methods with conic constraints are considered.

\subsection{Contribution of this work}

Previous approaches that have been proposed to solve problem ${\Psf}$ assume the total number of resources is constant. This critical assumption is impractical in most applications. For example, in power systems load demands are typically time-varying and the data defining such load demands may be uncertain \cite{Zhu2008}.  For this reason, any solution to network resource allocation problems should be robust to  uncertainty. This issue has not been addressed in the literature. Therefore, our main contribution in this paper is to address this question. In particular, motivated by our earlier work \cite{DoanB2016}, we design a distributed stochastic Lagrangian method to solve ${\Psf}$ when the constants $b_i$ are unknown and may be time-varying. We further provide an upper bound on the convergence rate of the method in expectation on the size and topology of the underlying networks.

Specifically, our primary contributions are summarized as follows.
\begin{itemize}[leftmargin = 3mm]
\item We first study a distributed Lagrangian method for problem ${\Psf}$, where the total quantity of resource is assumed to be constant over time-varying networks, thus generalizing our preliminary results as presented in \cite{DoanB2016}. The development and analysis in this case allows for an extension to the case where this quantity is uncertain. 

\item We then propose a distributed stochastic Lagrangian approach for problem ${\Psf}$ for the case where the constants $b_i$ are unknown and may be time-varying. We show that our approach is robust to this uncertainty, that is, our stochastic method achieves an asymptotic convergence in expectation to the optimal value. Moreover, we show that our method converges with rate $\mathcal{O}(n\ln(k)/\delta\sqrt{k})$, where $\delta$ is a parameter representing spectral properties of the graph structure underlying the connectivity of the nodes, $n$ is the number of nodes, and $k$ is the number of iterations. 
\end{itemize}
In addition, to illustrate the effectiveness of the proposed methods we present numerical results from applications to economic dispatch problems using the benchmark IEEE-14 and IEEE-118 bus test systems for three case studies.

We note that distributed algorithms are proposed in \cite{Kar2012} aimed at solving the economic dispatch problem under time-varying loads.  In our results herein, we allow for the functions $f_i$ to be more general (convex, rather than quadratic), and for the loads to be of a stochastic nature.  More importantly, using the algorithm we propose, we further provide the aforementioned convergence rates, where no such rates have been proven for earlier approaches.

The remainder of this paper is organized as follows.  We study a distributed Lagrangian method for problem ${\Psf}$ in Section \ref{sec:DLM}, while a stochastic version of this method is given for ${\Psf}$ under uncertainty in Section \ref{sec:DSLM}. In Section \ref{sec:Simulation}, we apply our proposed method to the economic dispatch problem on the benchmark IEEE-14 and IEEE-118 bus test systems, specifically for three case studies. We conclude the paper with a brief discussion of potential extensions in Section \ref{sec:Conclusion}. Finally, for an ease of exposition we provide our technical analysis for Sections \ref{sec:DLM} and \ref{sec:DSLM} in the supplementary document.  

\begin{notation}
We use boldface to distinguish between vectors $\mathbf{x}$ in $\mathbb{R}^n$ and scalars $x$ in $\mathbb{R}$. Given $\mathbf{x}\in\mathbb{R}^n$, let $\|\mathbf{x}\|$ denote its Euclidean norm. Moreover, we denote by $\mathbf{1}$ and $\Ibf$ the vector whose entries are all $1$ and the identity matrix, respectively.

Given a nonsmooth convex function $f:\Rset\rightarrow\Rset$, we denote by $\partial f(x)$ its subdifferential estimated at $x$, i.e., $$\partial f(x) \triangleq \{g\in\Rset\,|\, f(y) \geq f(x) + g(y-x), \; \forall\, y\in\Rset \},$$ the set of subgradients of $f$ at $x$, which is nonempty. In this paper, we denote by $g(x)$ a subgradient of $f$ estimated at $x$. Moreover, $f$ is $C$-Lipschitz continuous if and only if 
\begin{align*}
|\,f(x)-f(y)\,| \leq C\,|\,x-y\,|,\quad \forall\; x,y\in \Rset. 
\end{align*}   
Also, the Lipschitz continuity of $f$ is equivalent to the condition that the subgradients of $f$ are uniformly bounded by $C$ \cite[Lemma 2.6]{ShalevShwartz2012}. 

Finally, given any function $f:\Xcal\rightarrow\Rset$, where $\Xcal$ is convex, we denote by $\tilde{f}$ its extended value function, i.e., 
\begin{align*}
\tilde{f}(x) = \left\{\begin{array}{ll}
f(x) & x\in\Xcal\\
\infty  & \text{else}.  
\end{array}\right.
\end{align*}  
The Fenchel conjugate $\tilde{f}^*$ of $\tilde{f}$ is then given by
\begin{align*}
\tilde{f}^*(u) &= \underset{x\in\Rset}{\sup}\,\{\,ux - \tilde{f}(x)\,\}= \underset{x\in\Xcal}{\max}\,\{\,ux - f(x)\,\},
\end{align*}
which is always convex.
\end{notation}

\begin{remark}
In this paper we focus on the scalar case, i.e., $x_i \in \Rset$.  We note that the results herein could be extended in a straightforward manner to the vector case, $x_i \in \Rset^d$, along similar lines as those given in \cite{DoanBB2017}.
\end{remark}

%% file: alg_DLM.tex
\section{Distributed Lagrangian Methods}\label{sec:DLM}

Lagrangian methods have been widely used to construct a decentralized framework for problem {\Psf}, where each node in the network only has partial knowledge of the objective function and the constraints; this approach requires a central coordinator to update and distribute the Lagrange multiplier to the nodes. In this section, we present an alternative approach that allows us to bypass the need for a central coordinator, that is, our proposed approach allows for a truly distributed implementation, leading to a more efficient algorithm. This development also informs our approach on distributed stochastic Lagrangian methods for network resource allocation problems, which is presented in Section \ref{sec:DSLM}.

\subsection{Main Algorithm}\label{subsec_DLM:alg}
We start this section by explaining the mechanics of our approach. In particular, consider the following Lagrangian function ${\Lcal:\Rset^n\times\Rset\rightarrow\Rset}$ of ${\Psf}$ 
\begin{align}
\Lcal(\xbf,\lambda) := \sum_{i=1}^n f_i(x_i) + \lambda\left(\sum_{i=1}^n (x_i - b_i)\right),\label{alg:Lagrangian}
\end{align} 
where $\lambda\in\Rset$ is the Lagrangian multiplier associated with the coupling constraint \eqref{prob:couple_const}. The dual  function $d:\Rset\rightarrow\Rset$ of problem ${\Psf}$ for a given $\lambda$ is then defined as
\begin{align}
d(\lambda) &:= \underset{\xbf\in\Xcal}{\text{min}}\left\{ \sum_{i=1}^n f_i(x_i) + \lambda \left(\sum_{i=1}^n (x_i-b_i)\right) \right\}\nonumber\\
&=  \sum_{i=1}^n \Big(\,\underset{x_i\in\mathcal{X}_i}{\text{min}}\Big\{ \,f_i(x_i) + \lambda x_i\Big\}\, \Big)-\lambda\sum_{i=1}^n b_i\nonumber\\
&=  \sum_{i=1}^n \Big(\, - \underset{x_i\in\mathcal{X}_i}{\text{max}}\Big\{ \, - f_i(x_i) - \lambda x_i\Big\}\, \Big)-\lambda\sum_{i=1}^n b_i\nonumber\\
&= \sum_{i=1}^n \Big(\,-\tilde{f}_i^*(-\lambda)-\lambda b_i\,\Big).\label{alg:dual}
\end{align} 
The dual problem of ${\Psf}$, denoted by ${\sf DP}$, is given by
\begin{align}
{\sf DP}\ : \ \underset{\lambda\in\Rset}{\text{max}} \left\{\sum_{i=1}^n \Big(\,-\tilde{f}_i^*(-\lambda)-\lambda b_i\Big)\,\right\},\nonumber	
\end{align} 
which is then equivalent to solving 
\begin{align}
\underset{\lambda\in\Rset}{\text{min}}\; q(\lambda) \triangleq \sum_{i=1}^n\, \underbrace{\tilde{f}_i^*(-\lambda)\ + \ \lambda b_i}_{=q_i(\lambda)}, \label{alg:mindual}
\end{align}
where each $q_i \ : \ \Rset\rightarrow\Rset$ is convex since $\tilde{f}_i^*$ is convex. Moreover, the subgradient $g_i$ of $q_i$ is given as \cite{Bertsekas2003}
\begin{align}
g_i(\tilde{x}_i) = b_i - \tilde{x}_i,\label{alg:subg_qi}
\end{align} 
where $\tilde{x}_i$ satisfies  
\begin{align*}
\tilde{x}_i \in \arg\min_{x\in\Xcal_i} f_i(x_i) + \lambda (x_i-b_i).
\end{align*}
In the sequel, we denote by $$\Xcal \triangleq \left(\Xcal_1\times\Xcal_2\times\ldots\times\Xcal_n\right).$$ We consider the following Slater's condition to guarantee for the strong duality of problem {\sf P}.  

\begin{assump}[Slater's condition \cite{Bertsekas2003}]\label{assump:Slater}
There exists a point $\tilde{\xbf}$ in the relative interior of $\Xcal$ such that $\sum_{i=1}^n \tilde{x}_i = b$.
\end{assump}

For solving {\sf P}, standard Lagrangian methods \cite{Bertsekas2003} apply subgradient methods to update $\lambda$, which require a central coordinator to collect all the subgradients $g_i$ of the functions $q_i$. The nodes then use $\lambda$ distributed by this central coordinator, to update their variables $x_i$. However, such a central coordinator is not allowed in distributed frameworks, which motivates us to study distributed variants of Lagrangian methods. 

Indeed, the key idea of our approach is to eliminate the requirement of the central coordinator by utilizing the distributed consensus-based subgradient method presented in \cite{Nedic2009} to compute the solution of Eq.\ \eqref{alg:mindual}. In particular, we have each node $i$ stores a local copy $\lambda_i$ of $\lambda$. Each node $i$ then iteratively updates $\lambda_i$ upon communicating with its neighbors, where the goal of the nodes is to drive every $\lambda_i$ to a solution of {\sf DP} (a.k.a Eq.\ \eqref{alg:mindual}). In addition, each node $i$ utilizes its $\lambda_i$ to update its primal variable $x_i$, resulting in the distributed Lagrangian method, formally presented in Algorithm \ref{alg:DLM}. Here, Eq.\ \eqref{DLM:stepConsensus} is often referred to as a consensus step, while the term $b_i - x_i(k)$ in Eq.\ \eqref{DLM:stepSubgradient} is a ``local subgradient" of $q_i$ at $x_i(k)$ in Eq.\ \eqref{alg:subg_qi}. In addition, $a_{ij}(k)$ are the weights which node $i$ assigns for $\lambda_j$ received from node $j$ at time $k$. 

The updates in Algorithm \ref{alg:DLM} have a simple implementation: first, at time $k\geq 0$, each node $i$  broadcasts the current value of $\lambda_i(k)$ to its neighbors. Node $i$ computes $v_i$, as given by the weighted average of the current local copies of its neighbors and its own value. The updates of $x_i(k+1)$ and $\lambda_i(k+1)$ then do not require any additional communications among the nodes at step $k$. Finally, we note that our method maintains the feasibility of the nodes' local constraints at every iteration, i.e., $x_i(k)\in\mathcal{X}_i$ for all $k\geq 0$.

\begin{algorithm}
	\caption{Distributed Lagrangian Method for solving {\Psf}}
	\label{alg:DLM}
		1. \textbf{Initialize}: Each node $i$ initializes $\lambda_i(0) \in \Rset$.\\
2. \textbf{Iteration}: For $k\geq 0$, each node $i$ executes
\begin{align}
&v_i(k+1) = \sum_{j=1}^n a_{ij}(k) \lambda_j (k)  \label{DLM:stepConsensus}\\
&x_i(k+1) \in \arg\underset{x_i\in \Xcal_i}{\min}  f_i(x_i) + v_i(k+1) (x_i - b_i)\label{DLM:stepMinimize}\\
&\lambda_i(k+1) =  v_i(k+1)  + \alpha(k) \left(x_i (k+1)  -  b_i \right).\label{DLM:stepSubgradient}
\end{align}
\end{algorithm}

Regarding the network topology and inter-node communications, we assume that each node is only allowed to interact with neighbors that are directly connected to it through a sequence of time-varying undirected graphs. Specifically, we assume we are given a sequence of  undirected graphs $\mathcal{G}(k) = (\mathcal{V},\mathcal{E}(k))$ with $\mathcal{V} = \{1,\ldots,n\}$; nodes $i$ and $j$ can exchange messages at time $k$ if and only if $(i,j) \in \mathcal{E}(k)$. Denote by $\mathcal{N}_i(k)$ the neighboring set of node $i$ at time $k$. We make the following fairly standard assumption, which ensures the long-term connectivity of the network.
\begin{assump}\label{assump:BConnectivity}
There exists an integer $B\geq1$ such that the following graph is connected for all integers $\ell\geq0:$ 
\begin{equation}
(\mathcal{V},\mathcal{E}(\ell B)\cup \mathcal{E}(\ell B+1)\cup\ldots\cup \mathcal{E}((\ell +1)B-1)).
\end{equation}
\end{assump}
Intuitively this assumption ensures that all nodes can influence each other through repeated interactions with neighbors in the graph sequence $\mathcal{G}(k)$. We note this assumption is considerably weaker than requiring each $\mathcal{G}(k)$ to be connected for all $k\geq0$. In addition, we denote by $\Abf(k)$ the matrix whose $(i,j)$-th entries are the weights $a_{ij}(k)$ given in Eq.\ \eqref{DLM:stepConsensus}. We assume that $\Abf(k)$, which captures the topology of $\Gcal(k)$, satisfies the following conditions.
\begin{assump}\label{assump:DoubStocMatrix}
There exists a positive constant $\beta$ such that $\Abf(k)$ satisfies the following conditions for all $k\geq0$:
\begin{itemize}
\item[(a)] $a_{ii}(k)\geq\beta,$ for all $i$.
\item[(b)] $a_{ij}(k)\in[\beta,1]$ if $(i,j)\in\mathcal{N}_i(k)$ otherwise $a_{ij}(k)=0$ for all $i,j.$
\item[(c)] $\sum_{i=1}^n a_{ij}(k)=\sum_{j=1}^n a_{ij}(k)=1,$ for all $i,j$.
\end{itemize}
\end{assump}
In the sequel we denote by $\sigma_2(\Abf(k))$ the second largest singular value of $\Abf(k)$. Furthermore, let $\delta$ be a parameter representing the spectral properties of the graph defined as
\begin{align}
\delta \leq \min\,\left\{\,\left(1-\frac{1}{4n^3}\right)^{1/B}\,,\,\max_{k\geq 0}\sigma_2(\Abf(k))\,\right\}.\label{notation:delta}
\end{align}

\subsection{Convergence Analysis}\label{subsec_DLM:Analysis}
We are now ready to present our main results in this section on the convergence of Algorithm \ref{alg:DLM}. We denote by $\Lcal_i:\Rset\times\Rset\rightarrow\Rset$ the local Lagrangian function at node $i$ defined as
\begin{align}
\Lcal_i(x_i,v_i) = f_i(x_i) + v_i(x_i-b_i). 
\label{appendix:iLagrangian}
\end{align}
Our first result considers a particular sequence of stepsize $\{\alpha(k)\}$ that guarantees that the sequence $\{\lambda_i(k)\}$ for each node $i$ converges to a dual solution of ${\sf DP}$. This result is formally stated in the following theorem. 

\begin{thm}\label{alg_DLM:thmAsymConv}
Let Assumptions \ref{assump:Slater} --  \ref{assump:DoubStocMatrix} hold. Let the sequences $\{x_i(k)\}$ and $\{\lambda_i(k)\}$, for all $ i \in\Vcal$,  be generated by Algorithm \ref{alg:DLM}. Assume that the stepsize $\alpha(k)$ is non-increasing, with $\alpha(0) = 1$, and satisfies the following conditions:
\begin{align}
\sum_{k=1}^\infty \alpha(k) = \infty,\quad \sum_{k=1}^\infty \alpha^2(k) < \infty.\label{alg_DLM:thmStepsize}
\end{align}
Then the sequences $\{x_i(k)\}$ and $\{\lambda_i(k)\}$, for all $i\in\Vcal$, satisfy
\begin{enumerate}
\item[(a)] $\underset{k\rightarrow\infty}{\lim}\lambda_i(k) = \lambda^*$ is an optimizer of ${\sf DP}$. 
\item[(b)] $\underset{k\rightarrow\infty}{\lim}\sum_{i=1}^n \Lcal_i(x_i(k),\lambda_i(k))$ is the optimal value of ${\Psf}$.
\end{enumerate}
\end{thm}
A specific choice for the stepsize sequence is $\alpha(k) = 1/(k+1)$ for $k\geq 0$, which obviously satisfies \eqref{alg_DLM:thmStepsize}. Part (a) of Theorem \ref{alg_DLM:thmAsymConv} is a consequence of \cite[Proposition 4]{Nedic2010} while part (b) can be derived using the strong duality of ${\Psf}$. We present the proof of part (b) in the supplementary document. 

A key step in showing part $(a)$ requires that the subgradients of the dual function $q_i$ remain bounded for all $k\geq0$. Given the compactness of $\Xcal_i$ and Eq.\ \eqref{alg:subg_qi}, this boundedness condition is satisfied; formally stated as follows.

\begin{lem}\label{proposition:subgbound}
Let the sequences $\{x_i(k)\}$ and $\{\lambda_i(k)\}$, for all $i\in\Vcal$, be generated by Algorithm \ref{alg:DLM}. Then the subgradient $g_i(v_i(k))$ of $q_i(v_i(k))$ is bounded by some constant $C_i >0$
\begin{align}
|\,g_i(v_i(k))\,|\leq C_i,\quad \text{ for all }  i\in\Vcal.\label{alg_DLM:boundSG}
\end{align}
\end{lem}

We note that while the local copies of the dual variable $\lambda_i(k)$ tend to a dual optimizer $\lambda^*$ of the dual problem ${\sf DP}$, Theorem \ref{alg_DLM:thmAsymConv} does not automatically imply that
\begin{equation*}
\xbf(k) := \left(x_1^T(k),\ldots,x_n^T(k)\right)^T
\end{equation*} 
converges to a solution of ${\Psf}$. Such a convergence is guaranteed, however, when the functions $f_i$ are strongly convex. We remark that convergence to the optimal value without requiring explicit convergence of the optimizer is reminiscent of the behavior of centralized subgradient descent algorithms.

Our next result is to estimate how fast Algorithm \ref{alg:DLM} converges given some suitable choice of stepsizes. We exploit known techniques for the analysis of centralized subgradient methods to answer this question. In particular, if every node $i$ maintains a variable to track a time-weighted average of its dual variable, the distributed Lagrangian method converges at a rate $\mathcal{O}(nC\ln(k)/\delta\sqrt{k})$ when the stepsize decays as\footnote{Note that the choice of $\alpha(k) = 1/\sqrt{k+1}$ does not satisfy \eqref{alg_DLM:thmStepsize}. Hence, we only establish the rate of convergence to the optimal value.} $\alpha(k) = 1/\sqrt{k+1}$ and $C=\sum_{i=1}^n C_i$. We will use $q^*$ to denote the value of $q$ in \eqref{alg:mindual} evaluated at an optimal solution of the ${\sf DP}$. The following Corollary, is a consequence of \cite[Proposition 3]{Nedic2009}. We skip its proof and refer readers to \cite{Nedic2009, Nedic2010, Nedic2015} for a complete analysis.

\begin{cor}[\!\!\!\cite{Nedic2009}]\label{alg_DLM:convrate}
Let Assumptions \ref{assump:Slater} --  \ref{assump:DoubStocMatrix} hold. Let the sequences $\{x_i(k)\}$ and $\{\lambda_i(k)\}$, for all $i\in\Vcal$, be generated by Algorithm \ref{alg:DLM}. Let $\alpha(k) = 1/\sqrt{k+1}$ for $k\geq 0$. Moreover, suppose that every node $i$ stores a variable $y_i(k)\in\mathbb{R}$ arbitrarily initiated and updated by
\begin{align}
y_i(k+1) = \frac{\alpha(k)\lambda_i(k) + S(k)y_i(k)}{S(k+1)},\quad \text{ for } k\geq 0,\label{alg_DLM:yiUpdate}
\end{align}
where $S(0) = 0$ and $S(k+1) = \sum_{t=0}^{k}\alpha(t)$ for $k\geq 1$. Then, given an optimal $\lambda^*$ of {\sf DP}, we have for all $i\in \Vcal$ and  $k\geq 0$
\begin{align}
q(y_i(k)) - q^*&\leq \frac{n\Big(\bar{\lambda}(0)-\lambda^*\Big)^2}{2\sqrt{k+1}}+\frac{3C\|\bld(0)\|}{(1-\delta)\sqrt{k+1}}\notag\\
&\qquad +\frac{7C^2(1+\ln(k+1))}{2(1-\delta)\sqrt{k+1}} \cdot\label{alg_DLM:functionrate}
\end{align}
\end{cor}

Corollary \ref{alg_DLM:convrate} reveals that $q$ evaluated at time-averaged $\alpha$-weighted local copies of the dual variables (for each node) converges to the optimal value $q^*$. Moreover, it shows that the difference of $q$ at this `average' $\lambda$ from $q^*$ scales as $\mathcal{O} \left( {\ln(k)/\sqrt{k}} \right)$.
The $1/\sqrt{k}$ term mirrors the convergence results for centralized subgradient descent algorithms. For example, see \cite[Chapter 3]{Nesterov04}; essentially, the distributed nature of the algorithm slows the convergence by a factor of $\ln(k)$. Further, the convergence rate depends inversely on $1-\delta$, the \textit{spectral gap} of $\Abf(k)$ for $k\geq0$. In a sense, $\delta$ indicates how fast the information among the nodes is diffused across the graph. Finally, \cite{Nedic2015} shows that the spectral gap scales inversely with $n^3$ where $n$ is the number of agents. Thus, the larger the number of agents, the smaller the spectral gap, and hence, the slower the convergence.

%% file: alg_DRLM.tex

\section{Distributed Stochastic Lagrangian Methods}\label{sec:DSLM}
We now study ${\Psf}$ under uncertainty, where the portion of the resources is unknown. We are motivated by the fact that in real applications resource allotments are often changing over time, typically randomly, and their data may be uncertain. As an example, in power systems,  power loads - especially residential loads - fluctuate randomly due to the variation in energy consumption by consumers \cite{Zhu2008}. These random fluctuations also may result from the generation side, due to the variability in output from renewable resources. Thus, any solution to resource allocation problems should demonstrate some robustness to uncertainty.  Our goal in this section is to design a distributed stochastic Lagrangian method, and demonstrate that this method is robust to resource uncertainty. Motivated by the analysis in Section \ref{subsec_DLM:Analysis}, we also provide an upper bound for the rate of convergence of this method in expectation on the topology of the underlying networks.

\subsection{Main Algorithm}\label{subsec_DSLM:alg}
We assume the exact allotment of the resource is unknown and we can only estimate it from noisy data. For example, power generation levels in power systems at any time are predicted from hourly day-ahead energy consumption data, which may not be accurate. Therefore, we assume that at any time $k\geq0$ each node $i$ is able to access only a noisy measurement of $b_i$, i.e., node $i$ can sample $\ell_i(k)$ given as
\begin{align}
\ell_i(k) = b_i + \eta_i(k),\quad k\geq0,\label{alg_DSLM:loadmeasurement}
\end{align}
and the random variables $\eta_i$ represent random fluctuations in the allocations of the resources at the nodes; the sum of constants $b_i$ represents the expected resource shared by the nodes. We note that we do not assume the constants $b_i$ are known by the nodes, thus our model is general enough to cover the case of time-varying resources. We do, however, assume that the random variables $\eta_i$ satisfy the following assumption.

\begin{assump}\label{assump:boundednoise}
The random variables $\eta_i$ are independent with zero mean, i.e., $\mathbb{E}[\eta_i] = 0$, for all $i \in\Vcal$. Moreover, we assume that these random variables are almost surely bounded, i.e., for every $i$, there is a scalar $c_i > 0 $ such that $|\eta_i|\leq c_i$ almost surely for all $i\in\Vcal$.
\end{assump}

This assumption implies that we only allow finite, but possibly arbitrarily large, perturbations limited by the constants $c_i$, in the nodes' measurements. This condition is reasonable, for example, in actual power systems the hourly day-ahead data is often approximately accurate with respect to the current consumption. Moreover, small fluctuations in loads are often seen in practice since large fluctuations may lead to a blackout condition. We note that the assumption of zero mean implies that while being robust to the noisy measurements of the resources, the goal is to meet the expected number of loads defined by $\sum_i b_i$. Finally, here we do not make any assumption on the distribution of $\eta_i$.

We now proceed to present our distributed stochastic Lagrangian method for solving ${\Psf}$ under uncertainty in the constraint \eqref{prob:couple_const}, i.e., when the number of resources is unknown. Recall from Section \ref{subsec_DLM:alg} that in the distributed Lagrangian method we utilize the distributed subgradient algorithm to solve the dual ${\sf DP}$ of ${\Psf}$. Here, due to the uncertainty the nodes $i$ have to use a noisy measurement of $b_i$ represented by $\ell_i(k)$ to update their dual variables $\lambda_i(k)$, resulting in a distributed stochastic subgradient method for ${\sf DP}$. The proposed distributed stochastic Lagrangian algorithm is formally presented in Algorithm \ref{alg:DSLM}. 

Algorithm \ref{alg:DSLM} shares similar mechanics to Algorithm \ref{alg:DLM} studied in Section \ref{subsec_DLM:alg}. A notable difference is in Eq.\ \eqref{DSLM:stepSubgradient} of Algorithm \ref{alg:DSLM} where $\lambda_i$ is updated by moving a distance of $\alpha(k)$ along a noisy subgradient of $q_i$ at $v_i(k+1)$, given by
\begin{equation}
g_i^{s}(v_i(k+1)) \triangleq \ell_i(k) - x_i(k+1).\label{alg_DSLM:isubgradient}
\end{equation}  
We note that the variables $v_i$, $x_i,$ and $\lambda_i$ are now random variables because of the noisy measurements $\ell_i$.

\begin{algorithm}
	\caption{Distributed Stochastic Lagrangian Method  ({\sf DSLM}) for Solving {\Psf} under Uncertainty.}\label{alg:DSLM}
		1. \textbf{Initialize}: Each node $i$ initializes $\lambda_i(0) \in \Rset$\\
2. \textbf{Iteration}: For $k\geq 0$, each node $i$ executes 
\begin{align}
v_i(k+1) &= \sum_{j=1}^n a_{ij}(k) \lambda_j (k) \label{DSLM:stepConsensus}\\
			 x_i(k+1) &\in \arg\underset{x_i\in \Xcal_i}{\min}  f_i(x_i) + v_i(k+1) (x_i - \ell_i(k)) \label{DSLM:stepMinimize}\\
			\lambda_i(k+1) &=  v_i(k+1)  + \alpha(k) \left(x_i (k+1)  -  \ell_i(k) \right) \label{DSLM:stepSubgradient}
\end{align}
\end{algorithm}

\subsection{Convergence analysis}\label{subsecDSLM:ConvAnal}
We now present our main results on the convergence of Algorithm \ref{alg:DSLM}. For the sake of clarity, we again present the full analysis of these results in the supplementary document. Our first result considers a particular stepsize sequence $\{\alpha(k)\}$ that guarantees the sequence of the local copy of the dual variable $\{\lambda_i(k)\}$ for each node $i$ converges to a dual solution of ${\sf DP}$ almost surely ($a.s.$). This result, which can be viewed as a stochastic version of Theorem \ref{alg_DLM:thmAsymConv}, is formally stated in the following theorem. 

\begin{thm}\label{alg_DSLM:thmAsymConv}
Let Assumptions \ref{assump:Slater} -- \ref{assump:boundednoise} hold. Let the sequences $\{x_i(k)\}$ and $\{\lambda_i(k)\}$, for all $i\in\Vcal$, be generated by Algorithm \ref{alg:DSLM}. Assume that the step size $\alpha(k)$ is non-increasing, $\alpha(0) = 1$, and satisfies the following conditions,
\begin{align}
\sum_{k=1}^\infty \alpha(k) = \infty,\quad \sum_{k=1}^\infty \alpha^2(k) < \infty.\label{alg_DSLM:thmStepsize}
\end{align}
Then the sequences $\{x_i(k)\}$ and $\{\lambda_i(k)\}$, for all $i\in\Vcal$, satisfy
\begin{enumerate}[leftmargin = 5.5mm]
\item[(a)] $\underset{k\rightarrow\infty}{\lim}\lambda_i(k) = \lambda^*$ $a.s.$, where $\lambda^*$ is an optimizer of ${\sf DP}$. 
\item[(b)] $\underset{k\rightarrow\infty}{\lim}\Eset\Big[\sum_{i=1}^n \Lcal_i(x_i(k),\lambda_i(k))\Big]$ is the optimal value of ${\Psf}$.
\end{enumerate}
\end{thm}
Note that part (b) of Theorem \ref{alg_DSLM:thmAsymConv} can be derived using the strong duality of ${\Psf}$ and ${\sf DP}$. On the other hand, part $(a)$ is more involved compared to its deterministic counterpart in Theorem \ref{alg_DLM:thmAsymConv}. The key step to show part $(a)$ again requires that $g_i^s(v_i(k))$ remains bounded almost surely for all $k\geq0$. Fortunately, based on the compactness of $\Xcal_i$ and Assumption \ref{assump:boundednoise}, this condition is satisfied as given in the following lemma.

\begin{lem}\label{alg_DSLM:subgbound}
Let Assumption \ref{assump:boundednoise} hold. Let the sequences $\{x_i(k)\}$ and $\{\lambda_i(k)\}$, for all $i\in\Vcal$, be generated by Algorithm \ref{alg:DSLM}. Then there exists a positive constant $D_i$ such that,
\begin{align}
|\,g_i^s(v_i(k))\,|\leq D_i \;\text{ a.s.},\quad \text{ for all }  i\in\Vcal,\, k\geq0.\label{alg_DSLM:idualSB}
\end{align}
\end{lem}
Finally, we exploit the same technique as that used in Theorem \ref{alg_DLM:convrate} to establish the convergence rate of Algorithm \ref{alg:DSLM}.  We show this algorithm converges at a rate $\mathcal{O}(n\ln(k)/\delta\sqrt{k})$ to the optimal value in expectation. 
 
\begin{thm}\label{alg_DSLM:SDconvrate}
Let Assumptions \ref{assump:Slater}--\ref{assump:boundednoise} hold. Let $\{x_i(k)\}$ and $\{\lambda_i(k)\}$, for all $i\in\Vcal$, be generated by Algorithm \ref{alg:DSLM}. Let $\alpha(k) = 1/\sqrt{k+1}$ for $k\geq 0$. Moreover, suppose every node $i$ stores a variable $y_i\in\mathbb{R}$, arbitrarily initiated and updated by
\begin{align}
y_i(k+1) = \frac{\alpha(k)\lambda_i(k) + S(k)y_i(k)}{S(k+1)},\quad \text{ for } k\geq 0,\label{alg_DSLM:yiUpdate}
\end{align}
where $S(0) = 0$ and $S(k+1) = \sum_{t=0}^{k}\alpha(t)$ for $k\geq 0$. Then, we have for all $i\in\Vcal$
\begin{align}
\Eset\Big[q(y_i(k))\Big]- q^*&\leq \frac{n\Eset\Big[\Big(\bar{\lambda}(0)-\lambda^*\Big)^2\Big]}{2\sqrt{k+1}} + \frac{3D\Eset\Big[\|\bld(0)\|\Big]}{(1-\delta)\sqrt{k+1}} \notag\\
&\qquad +\frac{7D^2(1+\ln(k+1))}{2(1-\delta)\sqrt{k+1}}\cdot\label{alg_DSLM:functionrate}
\end{align}
\end{thm}

%% file: simulation.tex

\section{Case Studies}\label{sec:Simulation}
In this section, we consider case studies that demonstrate the effectiveness of the two methods proposed in Sections \ref{sec:DLM} and \ref{sec:DSLM}, for solving economic dispatch problems in power systems. We first consider the IEEE-14 bus test system \cite{IEEE14BusWebsite} where we consider two cases, constant loads and uncertain loads.  To test our methods on large scale systems, we then apply our method to the IEEE-118 bus test system \cite{IEEE118BusWebsite}, assuming a constant load. In all cases, we model the communication between nodes by a sequence of time-varying graphs. Specifically, we assume that at any iteration $k\geq 0$, a graph $\Gcal(k) = (\mathcal{V},\mathcal{E}(k))$ is generated randomly such that $\Gcal(k)$ is undirected and  connected. This implies that the connectivity constant $B$ in Assumption \ref{assump:BConnectivity} is equal to $1$. The sequence of adjacency matrices $\{\v{A}(k)\}$ is then set equal to the sequence of lazy Metropolis matrices corresponding to $\Gcal(k)$, i.e., for all $k\geq 0$,
\begin{align}
&\v{A}(k) = [a_{ij}(k)]\nonumber\\
&= \left\{\begin{array}{ll}
\frac{1}{2(\max\{|\mathcal{N}_i(k),\mathcal{N}_j(k)|\})}, &\!\!\! \text{if } (i,j) \in \mathcal{E}(k)\\
0, &\!\!\!\text{if } (i,j)\notin\mathcal{E} \text{ and } i\neq j\\
1-\sum_{j\in\mathcal{N}_i(k)}a_{ij}(k),&\!\!\! \text{if } i = j
\end{array}\right.
\end{align}
Note that since $\Gcal(k)$ represents undirected and connected graphs, it is obvious that $\v{A}(k)$ satisfies Assumption \ref{assump:DoubStocMatrix}. Finally, for all studies the simulations are terminated when the errors are less than $10\%$, i.e., $|\lambda_i(k)-\lambda^*|<0.1\;\lambda^*$ for all $i\in\Vcal$.

We now give an interpretation of problem {\Psf} in the context of economic dispatch problems. Here each node $i$ in the network can be interpreted as a single (or group of) generator(s) in one area. The variables $x_i$ represent the power generated at the generator(s) $i$. Upon generating an amount $x_i$ of power, generator $i$ suffers a cost as a function of its power, i.e., $f_i(x_i)$. The total cost of the network is then represented by the sum of the costs at each generator, i.e., $\sum_{i=1}^n f_i(x_i)$. The total power generated is required to meet the total load of the network described by the constant $P$, i.e., $\sum_{i=1}^n x_i = P$. This constraint implies that we are not concerned about the power balance at any individual generator, but only the overall balance in the network. Moreover, each generator $i$ is assumed to be able to generate only a limited power, i.e.,  $\ell_i\leq x_i\leq u_i$. The goal of this problem is to obtain an optimal load schedule for the generators while satisfying the network constraints. Comparative simulations of our proposed method to the approach in \cite{Kar2012} can be found in \cite{DoanB2016}.

\subsection{Economic dispatch for IEEE 14-bus test systems}

\begin{figure}[t]
\centering
\includegraphics[width=3in, keepaspectratio]{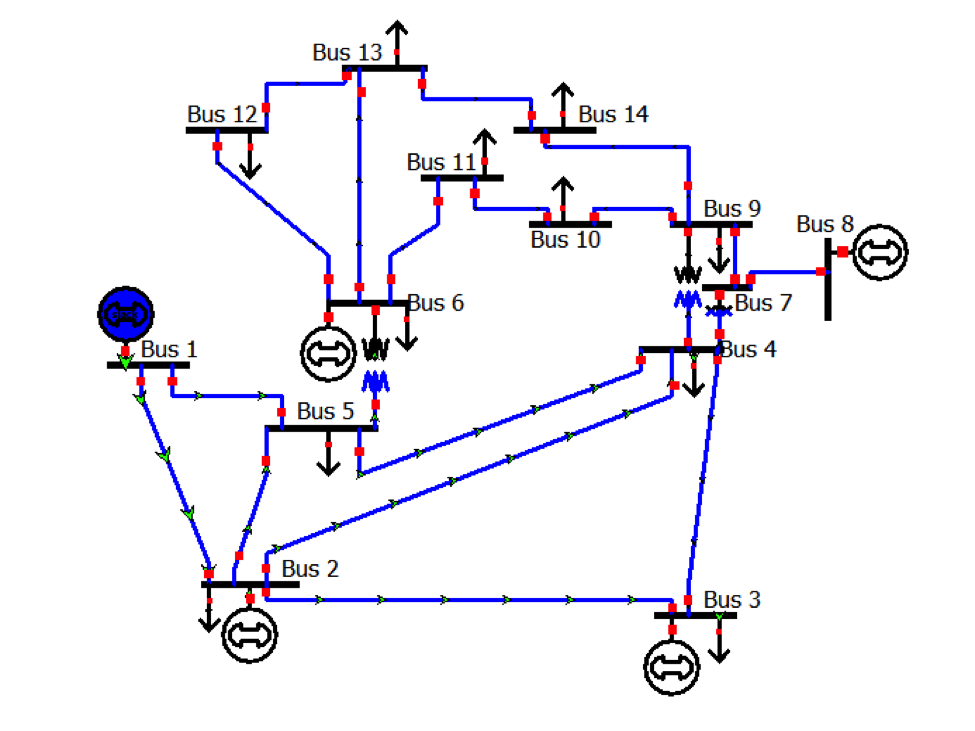}
\caption{IEEE $14$ bus systems.}\label{fig:IEEE14Bus}
\end{figure}

\begin{table}[h!]
\centering
\caption{node parameters (MU= Monetary units)} \begin{tabular}{|c|c|c|c|c|}
\hline
Gen. & Bus & $a_i[MU/MW^2]$ & $b_i[MU/MW]$ & $P_i^{\text{max}}[MW]$\\
\hline
 1 & 1 & 0.04 & 2.0 & 80 \\
 2 & 2 & 0.03 & 3.0 & 90 \\
 3 & 3 & 0.035 & 4.0 & 70 \\
 4 & 6 & 0.03 & 4.0 & 70 \\
 5 & 8 & 0.04 & 2.5 & 80 \\
\hline
\end{tabular}
\label{tab:14BusnodePar}
\end{table}

We now consider economic dispatch problems on the IEEE 14-bus test system \cite{IEEE14BusWebsite}. In this system, generators are located at buses $1,2,3,6,$ and $8$ as shown in Fig. \ref{fig:IEEE14Bus}. Each generator $i$ suffers a quadratic cost as a function of the amount of its generated power $P_i$, i.e., $f_i(P_i) = a_iP_i^2+b_iP_i$ where $a_i,b_i$ are cost coefficients of generators $i$. We also assume that each generator $i$ can only generate a limited amount of power constrained to lie in the interval $[0,P_i^{\text{max}}]$. The coefficients of the generators are listed in Table \ref{tab:14BusnodePar} which are adopted from \cite{Kar2012}. The expected load addressed by the network is assumed to be $P = 300 MW$. The goal now is to meet the load demand while minimizing the total cost of the nodes.

We first consider the case of constant loads, in which we initialize the generator power levels to $P_1 = 40(MW)$, $P_2 =  80(MW)$, $P_3 = 60(MW)$, $P_4 = 80(MW)$, and $P_5 =40(MW)$, giving $\sum_i P_i = P = 300(MW)$. We apply the distributed Lagrangian method of Algorithm \ref{alg:DLM} to solve this dispatch problem; simulations are shown in Fig. \ref{fig:IEEE14Bus_SimulationCL}. The top plot of Fig. \ref{fig:IEEE14Bus_SimulationCL} shows that our method achieves the optimal cost within a dozen iterations.  The consensus on the Lagrange multipliers $\lambda_i$ (i.e., the incremental cost) is shown in the center plot. Both plots verify our results, as presented in Section \ref{sec:DLM}. Finally, the bottom shows that the total generated power of the network, $\sum_i P_i$, meets the load demand $P=300(MW)$.

We then consider the case of uncertain loads, where we assume that at any iteration $k\geq 0$, each node $i$ has access to a noisy measurement of $b_i$, i.e., $b_i+\eta_i(k)$. Here each $\eta_i(k)$ is generated as independent zero-mean random variables. We apply the distributed stochastic Lagrangian method of Algorithm \ref{alg:DSLM} to this case. Simulations demonstrate the convergence of the total expected cost to the optimal cost; see the top plot in Fig. \ref{fig:IEEE14Bus_SimulationUL}.  In the same figure, the almost sure convergence of Lagrange multipliers is illustrated in the center plot, while the bottom plot shows that the total generated power meets the load demand, almost surely. As can be seen from this figure, the plots are influenced by the noise, as compared to the case of constant loads, i.e., the simulation results under load uncertainty are as not as smooth compared to those with constant loads. However, the convergence in expectation of the cost indicates that our method is robust to the load uncertainty.

\begin{figure}
\centering
    \begin{subfigure}[b]{\columnwidth}
        \centering
    \includegraphics[width=\columnwidth]{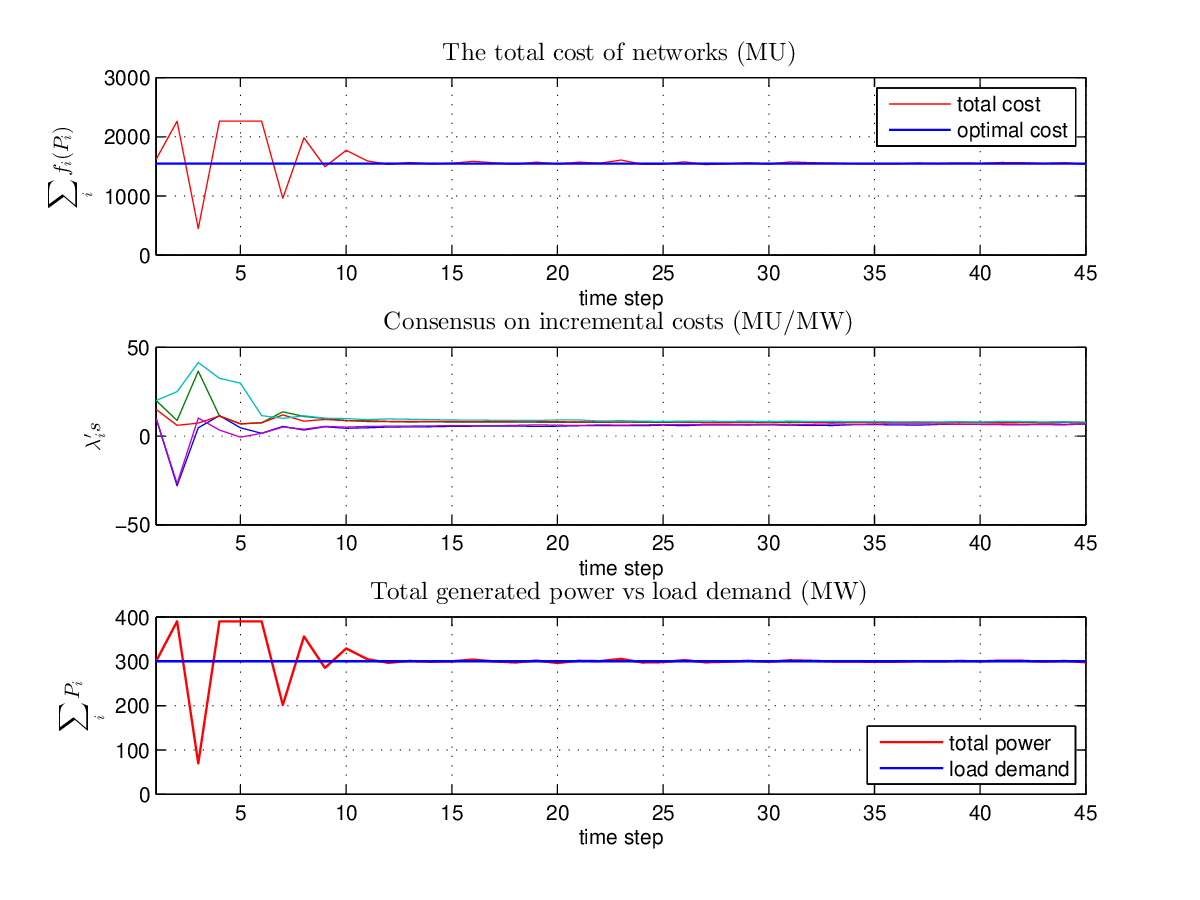}
    \caption{Constant loads.}
    \label{fig:IEEE14Bus_SimulationCL}
    \end{subfigure}
    \begin{subfigure}[b]{\columnwidth}
        \centering
        \includegraphics[width=\columnwidth]{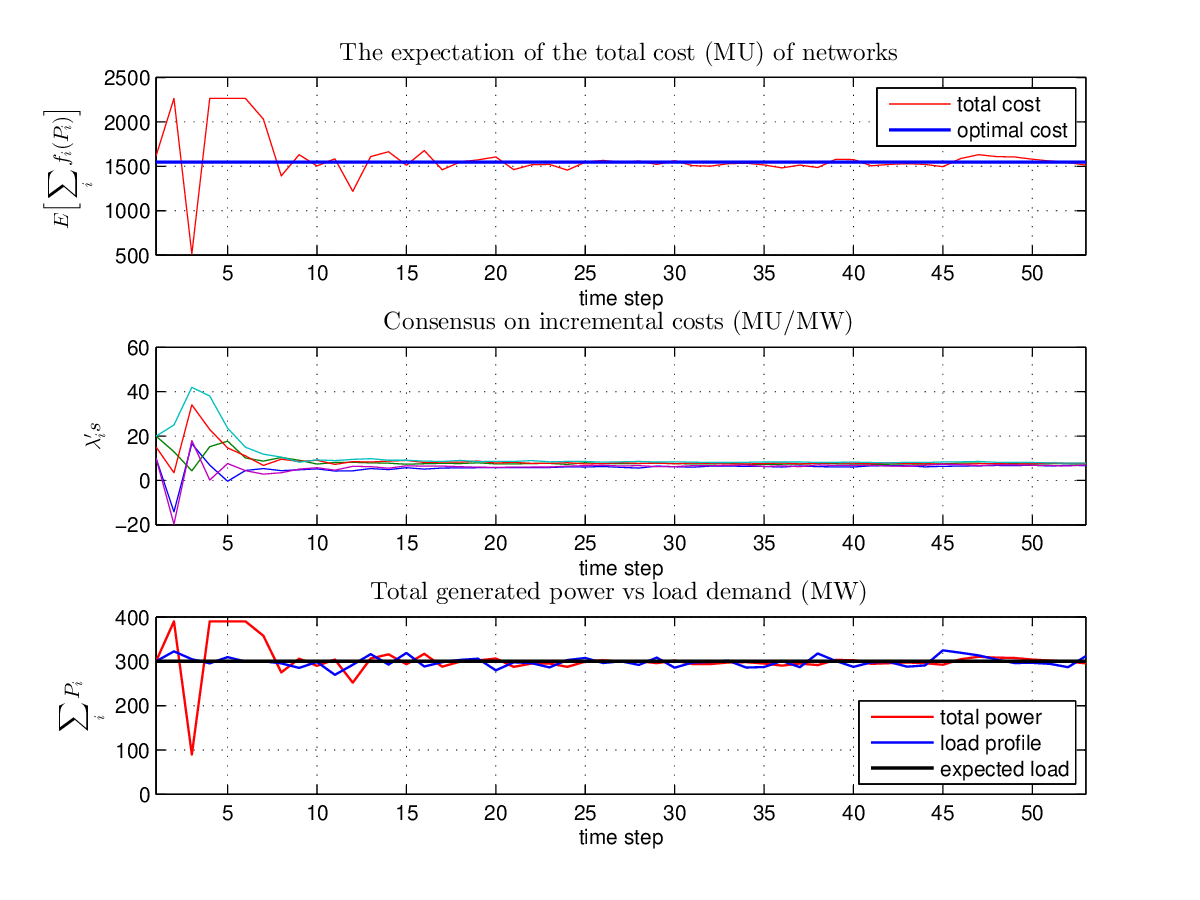}
        \caption{Load uncertainty.}
        \label{fig:IEEE14Bus_SimulationUL}
    \end{subfigure}
\caption{Simulation results for the economic dispatch problem: IEEE 14 bus test system.}
\label{fig:IEEE14Bus_Simulation}
\end{figure}



\subsection{Economic dispatch for IEEE 118-bus test systems}
We now consider economic dispatch problems on a larger system, the IEEE-118 bus test system \cite{IEEE118BusWebsite}. This system has $54$ generators connected by bus lines. Each generator $i$ suffers a quadratic cost as a function of generated power $P_i$, i.e., $f_i(P_i) = a_i + b_iP_i + c_i P_i^2$. The coefficients of functions $f_i$ belong to the ranges $a_i\in[6.78,74.33]$, $b_i\in[8.3391,37.6968]$, and $c_i\in[0.0024,0.0697]$. The units of $a,b,c$ are $MBtu, MBtu/MW$ and $MBtu/MW^2$, respectively. Each $P_i$ is constrained on some interval $[P_i^{\text{min}},P_i^{\max}]$ where these values vary as $P_i^{\text{min}}\in[5,150]$ and $P_i^{\text{max}}\in[150,400]$. The unit of power in this system is $MW$. The total load required from the system is assumed to be $P = 6000 (MW)$, which is initially distributed equally to the nodes, i.e., $P_i = P/54$ $\forall i\in\mathcal{V}$. We apply the distributed Lagrangian method for this study, with resulting simulations shown in Fig. \ref{fig:IEEE118_EDP}. The plots in Fig. \ref{fig:IEEE118_EDP} show the convergence of our method within $100$ iterations, implying that our method is applicable to large-scale systems.

\begin{figure}
\centering\vspace{-0.5cm}
\includegraphics[width=3.5in, keepaspectratio]{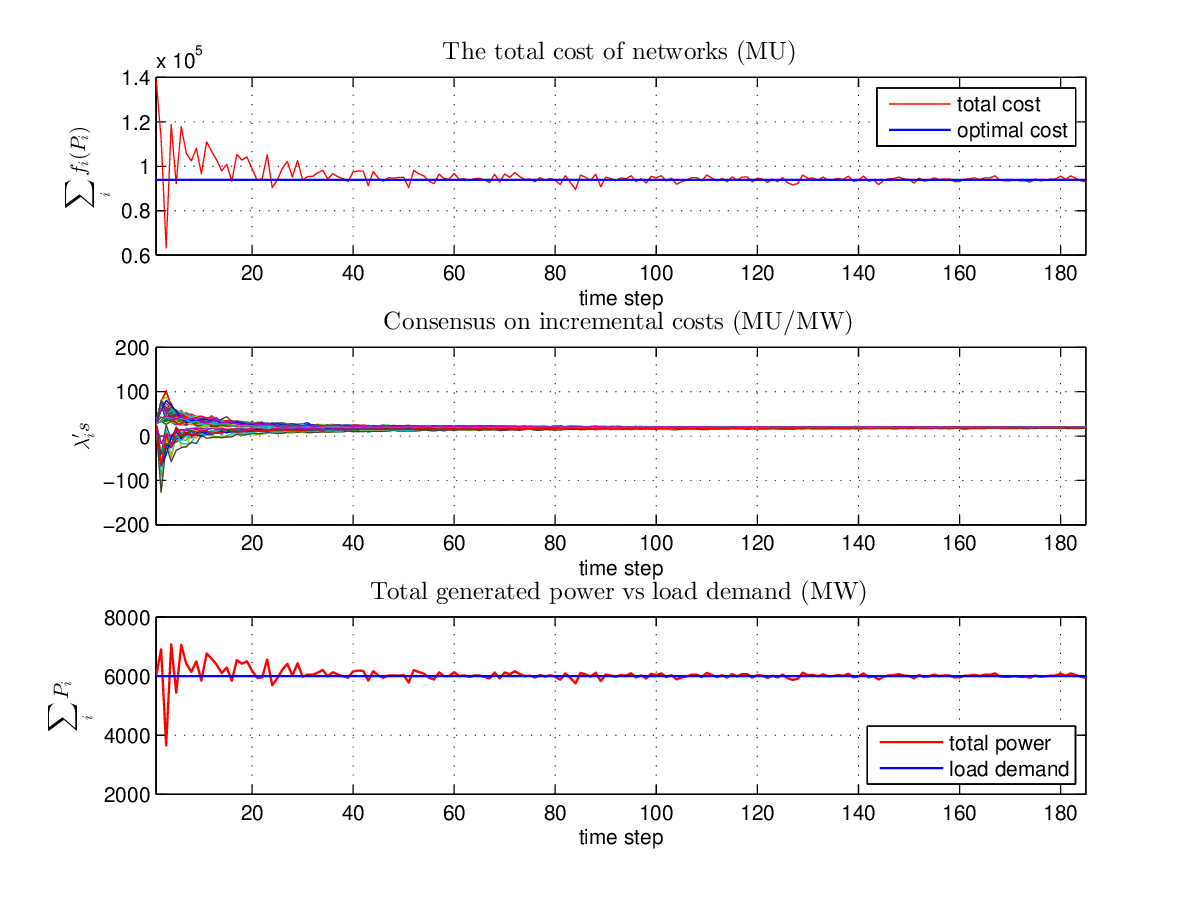}
\vspace{-1cm}
\caption{Simulation results for the economic dispatch problem: IEEE 118 bus test system.}\label{fig:IEEE118_EDP}
\end{figure}

%% file: proofs.tex

We provide here the proofs of main results from Sections \ref{sec:DLM} and \ref{sec:DSLM}. We start with some preliminaries and notation. 

\subsection{Preliminaries and Notations} 
In the sequel, we denote by $\mathcal{S}$ the feasible set of ${\Psf}$ $$\mathcal{S}=\{\mathbf{x}\in\mathcal{X} \ | \ \sum_{i=1}^n (x_i-b_i) = 0\}.$$ 
Let $f$ be defined as $f(\xbf) = \sum_{i=1}^n f_i(x_i)$.  Since $\mathcal{S}$ is compact and $f$ is continuous, there exists an optimal solution $\mathbf{x}^*=(x_1^*,x_2^*,\ldots,x_n^*)\in\mathcal{S}$ of {\sf P}. However, this solution is not unique. We denote the set of solutions of ${\Psf}$ as $\mathcal{S}^*$. Let $\lambda^*$ be an optimizer of ${\sf DP}$ and $\xbf^*$ be the corresponding optimizer of ${\Psf}$, such that $(\xbf^*,\lambda^*)$ is a saddle point of $\Lcal$ in \eqref{alg:Lagrangian}, i.e., 
\begin{align}
\Lcal(\xbf^*,\lambda)\leq\Lcal(\xbf^*,\lambda^*)\leq\Lcal(\xbf,\lambda^*),\; \forall \xbf\in \Xcal, \lambda\in\Rset.\label{appendix:saddlepoint}
\end{align}
Given a vector $\bld$, let
$\bar{\lambda} = \frac{1}{n}\sum_{i=1}^n \lambda_i.$ Finally, we review here an important result of the so-called distributed perturbed averaging algorithm studied in \cite{Nedic2015}. This result will play a key step in our analysis in the sequel. In particular, we consider a network of $n$ nodes which can exchange messages through a given sequence of time-varying undirected graphs $\Gcal(k) = (\mathcal{V},\mathcal{E}(k))$. Each node $i$ maintains a scalar variable $\lambda_i$ and  updates its variable as follows:
\begin{align}
\lambda_i(k+1) = \sum_{j\in \mathcal{N}_i} a_{ij}(k)\lambda_i(k) + \epsilon_i(k),\;\forall i\in\Vcal\label{appendix:ipertbAver}
\end{align}
where $\epsilon_i(k)$ is some disturbance at node $i$. By allowing $\epsilon_i(k)$ to take different forms, we can study different distributed algorithms; for example, this method reduces to distributed subgradient methods when $\epsilon_i(k) =  - \alpha(k)g_i(v_i(k))$ $\forall\, i\in\Vcal$.

We state here some important results which we will utilize in our development later. We first state a result on almost supermartingale convergence studied in \cite{Robbins1971} (see also in \cite{Polyak1987}, Lemma $11$, Chapter 2.2), which many refer to as Robbins-Siegmund Lemma. We then consider an important lemma of distributed perturbed averaging methods studied in \cite{Nedic2015}.
\begin{lem}[\cite{Robbins1971}]\label{appendix_lem:StochasticConv}
Let $\{y(k)\},$ $\{z(k)\}$, $\{w(k)\}$, and $\{\beta(k)\}$ be non-negative sequences of random variables and satisfy 
\begin{align}
&\mathbb{E}\Big[\,y(k+1)\,|\,\mathcal{F}_{k}\,\Big] \leq (1+\beta(k))y(k) - z(k) + w(k)\\
&\sum_{k=0}^\infty \beta(k) < \infty \text{ a.s, }\quad \sum_{k=0}^\infty w(k) < \infty \text{ a.s},
\end{align}
where $\mathcal{F}_k = \{y(0),\ldots,y(k)\}$, the history of $y$ up to time $k$. Then $\{y(k)\}$ converges a.s., and  $\sum_{k=0}^{\infty} z(k)<\infty$ a.s.
\end{lem}

\begin{lem}[\cite{Nedic2015}]\label{appendix_lem:pertbconsensus}
Let Assumptions \ref{assump:BConnectivity} and \ref{assump:DoubStocMatrix} hold. Let the sequence $\{\lambda_i(k)\}$, $k\geq 0$, be generated by Eq.\ \eqref{appendix:ipertbAver} with some $\lambda_i(0)\in \mathbb{R}$, for all $i\in\Vcal$. Then the following statements hold:
\begin{itemize}
\item[(1)] For all $i\in\Vcal$ and $k\geq 0$
\begin{equation}
\|\bld(k) - \bar{\lambda}(k)\1\|\leq \delta^{k}\|\bld(0)\| + \sum_{t=0}^{k}\delta^{k-t}\|\boldsymbol{\epsilon}(t)\|.\label{appendix_lem:averbound}
\end{equation}
\item[(2)]
Further if $\lim_{k\rightarrow\infty}\epsilon_i(k) = 0$ for all $i\in\Vcal$ then we have
\begin{align}
\lim_{k\rightarrow\infty} |\,\lambda_i(k) - \bar{\lambda}(k)\,| = 0\quad \text{ for all } i \in\mathcal{V}.\label{appendix_lem:averconv}
\end{align}
\item[(3)] Given a non-increasing positive sequence $\{\alpha(k)\}$ such that $~\sum_{k=0}^{\infty}\alpha(k)\|\boldsymbol{\epsilon}(k)\| <\infty$, then we obtain
\begin{align}
\sum_{k=0}^\infty \alpha(k)|\lambda_i(k)-\bar{\lambda}(k)| < \infty\quad \text{ for all} i\in\Vcal.\label{appendix_lem:finiteSum}
\end{align}
\end{itemize}
\end{lem}

\subsection{Proofs of Results in Section \ref{sec:DLM}}\label{appendix:DLM_thmAsym}

We present here the proof of part (b) in Theorem \ref{alg_DLM:thmAsymConv}; recall part (a) is a consequence of \cite[Proposition 4]{Nedic2010}. 
\begin{proof}[Proof of part (b) Theorem \ref{alg_DLM:thmAsymConv}]
Note that $(\xbf^*,\lambda^*)$ is a saddle point of $\Lcal$, i.e., $(\xbf^*,\lambda^*)$ satisfies \eqref{appendix:saddlepoint}. Eq. \eqref{appendix:iLagrangian} gives
\begin{align}
\sum_{i=1}^n\mathcal{L}_i(x_i,v_i) = \sum_{i=1}^n f_i(x_i) + v_i(x_i-b_i).\nonumber
\end{align}
To show our main result, we will show the following relation,
\begin{align}
&0\leq \sum_{i=1}^n\mathcal{L}_i(x_i^*,v_i(k))-\mathcal{L}_i(x_i(k),v_i(k))\nonumber\\
&\leq C\|\mathbf{v}(k)-\lambda^*\mathbf{1}\|  + \sum_{i=1}^nv_i(k)(x_i^*-b_i).\label{appendix_DLM:Eq1}
\end{align}
We note that by part (a) $\lim_{k\rightarrow\infty}\lambda_i(k)= \lambda^*$, for all $i\in\Vcal$, implying ${\lim_{k\rightarrow\infty}v_i(k) = \lambda^*}$ since $\v{A}(k)$ is doubly stochastic. In addition, the condition  $\sum_{i=1}^n (x_i^*- b_i) = 0$ gives
\begin{align}
&\lim_{k\rightarrow\infty}\sum_{i=1}^nv_i(k+1)(x_i^*-b_i) = 0.\nonumber
\end{align}
Thus, Eq.\ \eqref{appendix_DLM:Eq1} gives
\begin{align}
0&\leq\lim_{k\rightarrow\infty}\sum_{i=1}^n\Big\{ \Lcal_i(x_i^*,v_i(k))-\mathcal{L}_i(x_i(k),v_i(k))\Big\}\nonumber\\
&=\lim_{k\rightarrow\infty}\sum_{i=1}^n\Big\{\Lcal_i(x_i^*,\lambda^*)-\mathcal{L}_i(x_i(k),\lambda^*)\Big\} = 0,\nonumber
\end{align}
which by \eqref{alg:Lagrangian} and \eqref{appendix:iLagrangian} gives part (b), i.e.,
\begin{align}
&0\leq f(\mathbf{x^*}) - \lim_{k\rightarrow\infty}\sum_{i=1}^n\mathcal{L}_i(x_i(k),\lambda_i(k))= 0,\label{appendix_DLM:Eq2}
\end{align}
where we use $\mathcal{L}(\mathbf{x}^*,\lambda^*) = f(\xbf^*)$ due to the strong duality.

We now proceed to show \eqref{appendix_DLM:Eq1}. Since $x_i(k+1)$ satisfies Eq.\ \eqref{DLM:stepMinimize} and by the definition of $\mathcal{L}_i$ we have for any $k\geq 0$,
\begin{align}
0 \leq \mathcal{L}_i(x_i^*,v_i(k)) - \mathcal{L}_i(x_i(k),v_i(k)),\quad \forall i\in\mathcal{V},\nonumber
\end{align}
which when summing for all $i\in\mathcal{V}$ implies that
\begin{align}
&0\leq\sum_{i=1}^n\Big\{\Lcal_i(x_i^*,v_i(k))-\Lcal_i(x_i(k),v_i(k))\Big\}\nonumber\\
&= \sum_{i=1}^n\Big\{f_i(x_i^*) + v_i(k)(x_i^*-b_i)\Big\}\nonumber\\
&\qquad -\sum_{i=1}^n\Big\{f_i(x_i(k))+ v_i(k)(x_i(k)-b_i)\Big\}.\label{appendix_DLM:Eq5}
\end{align}
The strong duality and Eqs.\ \eqref{alg:dual} and \eqref{alg:mindual} imply
\begin{align}
\sum_{i=1}^n f_i(x_i^*) = d(\lambda^*) = -q(\lambda^*)= -\sum_{i=1}^nq_i(\lambda^*).\nonumber
\end{align}
Moreover, by Eqs.\ \eqref{DLM:stepMinimize} and \eqref{alg:mindual} we have
\begin{align}
q_i(v_i(k))=- f_i(x_i(k)) - v_i(k)(x_i(k)-b_i).\nonumber
\end{align}
Substituting the previous two preceding relations into Eq.\ \eqref{appendix_DLM:Eq5} and by Eq.\ \eqref{alg_DLM:boundSG} we obtain Eq.\ \eqref{appendix_DLM:Eq1}, i.e.,
\begin{align*}
0&\leq\sum_{i=1}^n\Big\{\mathcal{L}_i(x_i^*,v_i(k))-\mathcal{L}_i(x_i(k),v_i(k))\Big\}\nonumber\\
&= \sum_{i=1}^n \Big\{q_i(v_i(k)) - q_i(\lambda^*) + v_i(k)(x_i^*-b_i)\Big\}\nonumber\\
&\leq C\,\|\,\lambda^*\1-\vbf(k)\,\| + \sum_{i=1}^n v_i(k)(x_i^*-b_i).
\end{align*}
\end{proof}

\subsection{Proofs of Results in Section \ref{sec:DSLM}}
We provide here the proofs for main results presented in Section \ref{subsecDSLM:ConvAnal}. The key idea of our analysis is to study the convergence of distributed stochastic subgradient methods for solving $\sf DP$ over undirected graphs.  We note that distributed stochastic subgradient methods has been studied in \cite{Nedic2016a} for optimization problems defined over directed graphs with strongly convex objective functions. However, we consider here the case of convex objective functions. We, therefore, provide a convergence analysis of such methods, which is more straightforward than the one in \cite{Nedic2016a}.         

Let $\mathcal{L}_i^s:\Rset\times\Rset\times\Rset\rightarrow\Rset$ be the local stochastic Lagrangian function at node $i$ defined as,
\begin{align}
\mathcal{L}_i^s(x_i,v_i,\ell_i) = f_i(x_i) + v_i(x_i-\ell_i).\label{appendix:StociLagrangian}
\end{align}
We define $\mathcal{F}_k$ to be all the information generated by {\sf DSLM} up to time $k$, i.e., all the $x_i(k),\lambda_i(k)$ and so forth for $k\geq 0$. Let $D=\sum_{i=1}^nD_i$.  We start with the analysis of Theorem \ref{alg_DSLM:thmAsymConv}. 

\begin{proof}[Part (a)]
Let $\lambda^*$ be a minimizer of dual problem \eqref{alg:dual}. Using Eqs.\ \eqref{DSLM:stepConsensus} and \eqref{DSLM:stepSubgradient} in Algorithm \ref{alg:DSLM}, and Eq.\ \eqref{alg_DSLM:isubgradient} we obtain 
\begin{align*}
\bar{\lambda}(k+1) = \bar{\lambda}(k) - \frac{\alpha(k)}{n}\sum_{i=1}^n g_i^s(v_i(k+1)).
\end{align*}
The preceding equation implies that
\begin{align*}
&\Big(\bar{\lambda}(k+1) - \lambda^*\Big)^2\notag\\
&= \Big(\bar{\lambda}(k)-\lambda^*\Big)^2 + \frac{\alpha^2(k)}{n^2}\sum_{i=1}^n\Big[g_i^s(v_i(k+1))\Big]^2 \nonumber\\
&\qquad-\frac{2\alpha(k)}{n}\sum_{i=1}^n g_i^s(v_i(k+1))(\bar{\lambda}(k) - \lambda^*)\displaybreak[0]\notag\\
&\stackrel{\eqref{alg_DSLM:idualSB}}{\leq} \Big(\bar{\lambda}(k)-\lambda^*\Big)^2 + \frac{D^2\alpha^2(k)}{n}\nonumber\\
&\qquad-\frac{2\alpha(k)}{n}\sum_{i=1}^n g_i^s(v_i(k+1))(\bar{\lambda}(k) - \lambda^*),
\end{align*}
which by taking the conditional expectation w.r.t. $\Fcal_k$ yields
\begin{align}
&\Eset\Big[\big(\bar{\lambda}(k+1) - \lambda^*\big)^2\,|\,\Fcal_k\Big]= \Big(\bar{\lambda}(k)-\lambda^*\Big)^2 + \frac{D^2\alpha^2(k)}{n}\nonumber\\
&\qquad\qquad -\frac{2\alpha(k)}{n}\sum_{i=1}^n g_i(v_i(k+1))(\bar{\lambda}(k) - \lambda^*),\label{appendix_DSLM:Eq1}
\end{align}
where by Eqs.\ \eqref{alg_DSLM:loadmeasurement} and \eqref{alg_DSLM:isubgradient}, and Assumption \ref{assump:boundednoise} we have
$$\Eset\Big[\,g_i^s(v_i(k+1))\,|\,\mathcal{F}_k\Big] = g_i(v_i(k+1)) \in \partial q_i(v_i(k+1)).$$ 
Consider the last term on the right-hand side of Eq.\ \eqref{appendix_DSLM:Eq1}
\begin{align*}
&-\frac{2\alpha(k)}{n}\sum_{i=1}^n g_i(v_i(k+1))(\bar{\lambda}(k) - \lambda^*)\notag\\
&\quad = -\frac{2\alpha(k)}{n}\sum_{i=1}^n g_i(v_i(k+1))(\bar{\lambda}(k) -v_i(k+1))\notag\\
&\quad\qquad -\frac{2\alpha(k)}{n}\sum_{i=1}^n g_i(v_i(k+1))(v_i(k+1) - \lambda^*)\notag\\
&\quad \stackrel{\eqref{alg_DSLM:idualSB}}{\leq} \frac{2\alpha(k)}{n}\sum_{i=1}^n D_i\,|\,\bar{\lambda}(k) - v_i(k+1)\,|\notag\\
&\quad\qquad -\frac{2\alpha(k)}{n}\sum_{i=1}^n\Big( q_i(v_i(k+1)) - q_i(\lambda^*)\Big)\notag\\
&\quad\leq \frac{2D\alpha(k)}{n}\|\,\vbf(k+1) - \bar{\lambda}(k)\1\|\notag\\
&\quad\qquad -\frac{2\alpha(k)}{n}\sum_{i=1}^n\Big( q_i(v_i(k+1)) - q_i(\bar{\lambda}(k))\Big)\notag\\
&\quad\qquad - \frac{2\alpha(k)}{n}\sum_{i=1}^n\Big( q_i(\bar{\lambda}(k)) - q_i(\lambda^*)\Big)\notag\\
&\quad\stackrel{\eqref{alg_DSLM:idualSB}}{\leq}   \frac{4D\alpha(k)}{n}\|\,\bld(k) - \bar{\lambda}(k)\1\| - \frac{2\alpha(k)}{n}\Big( q(\bar{\lambda}(k)) - q^*\Big), 
\end{align*}
where we have used the following inequality 
\begin{align*}
\|\vbf(k+1) - \bar{\lambda}(k)\1\| \leq \|\bld(k) - \bar{\lambda}(k)\1\|.
\end{align*}
Substituting the equation above into Eq.\ \eqref{appendix_DSLM:Eq1} yields 
\begin{align}
&\Eset\Big[\big(\bar{\lambda}(k+1) - \lambda^*\big)^2\,|\,\Fcal_k\Big]\notag\\
&\leq \Big(\bar{\lambda}(k)-\lambda^*\Big)^2 + \frac{D^2\alpha^2(k)}{n}+ \frac{4D\alpha(k)}{n}\|\,\bld(k) - \bar{\lambda}(k)\1\|\nonumber\\
&\qquad - \frac{2\alpha(k)}{n}\Big( q(\bar{\lambda}(k)) - q^*\Big).\label{appendix_DSLM:Eq2}
\end{align}
Recall that Eq.\ \eqref{DSLM:stepSubgradient} in Algorithm \ref{alg:DSLM} is a special case of the perturbed averaging protocol Eq.\ \eqref{appendix:ipertbAver} where 
\begin{equation*}
\epsilon_i(k) = -\alpha(k)g_i(v_i(k+1)).
\end{equation*}
Since $\alpha(k)$ satisfies Eq.\ \eqref{alg_DSLM:thmStepsize} and by Eq.\ \eqref{alg_DSLM:idualSB} we obtain
\begin{align*}
\sum_{k=0}^\infty \alpha(k)\|\boldsymbol{\epsilon}(k)\| 
&\leq D\sum_{k=0}^\infty\alpha^2(k)<\infty \text{ a.s.,}
\end{align*}
which satisfies Eq.\ \eqref{appendix_lem:averbound} in Lemma \ref{appendix_lem:pertbconsensus}. Thus we obtain
\begin{align*}
\sum_{k=0}^\infty \alpha(k)\,|\,\lambda_i(k)-\bar{\lambda}(k)\,| < \infty \text{ a.s. for all } i\in\Vcal,
\end{align*}
which implies that
\begin{align}
\sum_{k=0}^\infty \alpha^2(k)+ \sum_{k=0}^\infty\alpha(k)\|\bld(k)-\bar{\lambda}(k)\1\| < \infty \text{ a.s.}\nonumber
\end{align}
Thus, applying Lemma \ref{appendix_lem:StochasticConv} to Eq.\ \eqref{appendix_DSLM:Eq2} yields
\begin{align*}
&\Big\{|\,\bar{\lambda}(k)-\lambda^*\,|\Big\} \text{ converges a.s. for each } \lambda^*,\\
&\sum_{k=0}^\infty \alpha(k) \Big(q(\bar{\lambda}(k))-q^*\Big) < \infty \text{ a.s.},
\end{align*}
which since $\sum_{k=0}^\infty\alpha(k)=\infty$ implies
\begin{align}
\liminf_{k\rightarrow\infty} q(\bar{\lambda}(k)) = q^* \text{ a.s.}\label{appendix_DSLM:Eq3c}
\end{align}
Let $\{\bar{\lambda}(k_{\ell})\}$ be a subsequence of $\{\bar{\lambda}(k)\}$ such that 
\begin{align*}
\lim_{\ell\rightarrow\infty} q(\bar{\lambda}(k_{\ell})) = \liminf_{k\rightarrow\infty} q(\bar{\lambda}(k)) = q^* \text{ a.s.}
\end{align*}
Since $\{|\,\bar{\lambda}(k)-\lambda^*\,|\}$ converges, the sequence $\{\bar{\lambda}(k_{\ell})\}$ is bounded. Hence, there is a convergent subsequence of $\{\bar{\lambda}(k_{\ell})\}$. Since $\lim_{\ell\rightarrow\infty} q(\bar{\lambda}(k_{\ell})) = q^*$ a.s., this subsequence converges to a minimizer $\tilde{\lambda}$ of {\sf DP} a.s. In addition, since $\Big\{|\,\bar{\lambda}(k)-\lambda^*\,|\Big\} \text{ converges a.s. for each } \lambda^*$, we obtain
\begin{align*}
\lim_{k\rightarrow\infty}\bar{\lambda}(k) = \lambda^*\text{ a.s. },
\end{align*}
which together with Eq.\ \eqref{appendix_lem:averconv} implies part (a).
\end{proof}

\begin{proof}[\textit{Part (b)}] Recall that $(\xbf^*,\lambda^*)$ is a saddle point of the Lagrangian \eqref{alg:Lagrangian}. Since $x_i(k)$ satisfies Eq.\ \eqref{DSLM:stepMinimize} and by using Eq.\ \eqref{appendix:StociLagrangian} we have for all $i\in\Vcal$  and $k\geq 0$
\begin{align*}
0 &\leq\,\Lcal_i^s(x_i^*,v_i(k),\ell_i(k-1)) - \Lcal_i^s(x_i(k),v_i(k),\ell_i(k-1)),
\end{align*}
which when summing over $i$ implies
{\small
\begin{align}
&0\leq\sum_{i=1}^n\mathcal{L}_i^s(x_i^*,v_i(k),\ell_i(k-1))\mathcal{L}_i^s(x_i(k),v_i(k),\ell_i(k-1))\nonumber\\
&= \sum_{i=1}^n f_i(x_i^*) + v_i(k)(x_i^*-\ell_i(k-1))\nonumber\\
&\qquad - \sum_{i=1}^n f_i(x_i(k)) + v_i(k)(x_i(k)-\ell_i(k-1)).\label{appendix_DSLM:Eq4}
\end{align}}
First, the strong duality implies
\begin{align}
\sum_{i=1}^n f_i(x_i^*) = -\sum_{i=1}^nq_i(\lambda^*).\label{appendix_DSLM:Eq5}
\end{align}
Second, recall from \eqref{alg:mindual} that
\begin{align}
q_i(v_i(k)) =  - f_i(x_i(k)) - v_i(k)(x_i(k)-b_i).\label{appendix_DSLM:Eq6}
\end{align}
Moreover, by Assumption \ref{assump:boundednoise} $\mathbb{E}[\ell_i(k)] = b_i$, for all $i\in\Vcal$. Taking the expectation of both sides in Eq.\ \eqref{appendix_DSLM:Eq4} with respect to $\mathcal{F}_{k-1}$ and using Eqs.\ \eqref{appendix_DSLM:Eq5} and \eqref{appendix_DSLM:Eq6} we have
\begin{align*}
&0\leq\sum_{i=1}^n\Eset\Big[\,\Lcal_i^s(x_i^*,v_i(k),\ell_i(k-1))\,|\,\Fcal_{k-1}\,\Big]\nonumber\\
&\quad\qquad- \Eset\Big[\,\Lcal_i^s(x_i(k),v_i(k),\ell_i(k-1))\,|\,\Fcal_{k-1}\,\Big]\nonumber\\
&=\mathbb{E}\Big[\, \sum_{i=1}^n q_i(v_i(k)) - q_i(\lambda^*)\,|\,\Fcal_{k-1}\Big]+ \sum_{i=1}^n v_i(k)(x_i^*-b_i)\nonumber\\
&\leq \sum_{i=1}^n D_i\, |\,\lambda^* - v_i(k)\,| + \sum_{i=1}^n v_i(k)(x_i^*-b_i)\notag\\
&\leq D\, \|\,\bld(k-1)-\lambda^*\1\,\| + \sum_{i=1}^n v_i(k)(x_i^*-b_i),
\end{align*}
which by taking the expectation and letting $k\rightarrow\infty$ we obtain
\begin{align}
0&\;\leq\lim_{k\rightarrow\infty}\sum_{i=1}^n\Eset\Big[\,\Lcal_i^s(x_i^*,v_i(k),\ell_i(k-1))\,\Big]\nonumber\\
&\quad\qquad-\lim_{k\rightarrow\infty} \Eset\Big[\,\Lcal_i^s(x_i(k),v_i(k),\ell_i(k-1))\,\Big]\nonumber\\
&\leq \lim_{k\rightarrow\infty}D\, \Eset\Big[\,\|\,\bld(k-1)-\lambda^*\1\,\|\Big]\notag\\ 
&\qquad\quad + \lim_{k\rightarrow\infty}\Eset\Big[\sum_{i=1}^n v_i(k)(x_i^*-b_i)\Big]\nonumber\\
&= 0 \ \ \text{a.s.} ,\label{appendix_DSLM:Eq7}
\end{align}
where the last equality is due to $$\lim_{k\rightarrow\infty}\sum_{i=1}^nv_i(k)(x_i^*-b_i) = \sum_{i=1}^n\lambda^*(x_i^*-b_i) = 0\text{ a.s.}$$ By Assumption \ref{assump:boundednoise}, and Eqs.\ \eqref{appendix:iLagrangian} and \eqref{appendix:StociLagrangian} we have 
\begin{align*}
&\lim_{k\rightarrow\infty}\sum_{i=1}^n\Eset\Big[\,\Lcal_i^s(x_i^*,v_i(k),\ell_i(k-1))\,\Big]\notag\\ 
&\qquad= \lim_{k\rightarrow\infty}\sum_{i=1}^n\Eset\Big[\,\Lcal_i(x_i^*,v_i(k))\,\Big] = f^*,
\end{align*}
where we use $\lim_{k\rightarrow \infty} v_i(k) = \lambda^*$  a.s. This together with Eq.\ \eqref{appendix_DSLM:Eq7} implies part (b).
\end{proof}

Finally, we present the proof of Theorem \ref{alg_DSLM:SDconvrate}.

\begin{proof}[Proof of Theorem \ref{alg_DSLM:SDconvrate}]
First, summing up both sides of Eq.\ \eqref{appendix_DSLM:Eq2} over $k=0,\ldots,K$ for some $K\geq 1$ we have
\begin{align}
&\Eset\Big[\big(\bar{\lambda}(K+1) - \lambda^*\big)^2\,|\,\Fcal_k\Big]\notag\\
&\leq \Big(\bar{\lambda}(0)-\lambda^*\Big)^2 + \frac{4D}{n}\sum_{k=0}^{K}\alpha(k)\|\,\bld(k) - \bar{\lambda}(k)\1\|\nonumber\\
&\qquad+ \frac{D^2}{n}\sum_{k=0}^{K}\alpha^2(k) - \frac{2}{n}\sum_{k=0}^{K}\alpha(k)\Big( q(\bar{\lambda}(k)) - q^*\Big).\label{appendix_DSLM:Eq8}
\end{align}
Next, using Eq.\ \eqref{appendix:ipertbAver} with $$\epsilon_i(k) = -\alpha(k)g_i(v_i(k+1))$$ and by Eq.\ \eqref{appendix_lem:averbound} we have for some $K\geq 1$
\begin{align*}
&\sum_{k=0}^K \alpha(k)\,|\,\bld(k)-\bar{\lambda}(k)\1\,|\nonumber\\
&\leq \|\bld(0)\|\sum_{k=0}^K \delta^k\alpha(k) + D\sum_{k=0}^K\alpha(k)\sum_{t=0}^{k}\delta^{k-t}\alpha(t),
\end{align*}
which since $\delta < 1$ and $\alpha(k)=1/\sqrt{k+1}$ for $k \geq 0$ we obtain
\begin{align}
&\sum_{k=0}^K \alpha(k)\,\|\,\bld(k)-\bar{\lambda}(k)\1\,\|\nonumber\\
&\qquad \leq \|\bld(0)\|\sum_{k=0}^K \delta^k+ D\sum_{k=0}^K\sum_{t=0}^{k}\delta^{k-t}\alpha^2(t)\notag\\
&\qquad \leq \frac{\|\bld(0)\|}{1-\delta}+ D\sum_{k=0}^K\sum_{t=0}^{k}\frac{\delta^{k-t}}{t+1}\nonumber\\
&\qquad = \frac{\|\bld(0)\|}{1-\delta} + D \sum_{t=0}^K\frac{1}{t+1}\sum_{\ell=0}^{K-t}\delta^{\ell}\nonumber\\
&\qquad \leq\frac{\|\bld(0)\|}{1-\delta}+ \frac{D(1+\ln(K+1))}{1-\delta},\label{appendix_DSLM:Eq9}
\end{align}
where the last inequality is due to the integral test
\begin{align}
\sum_{t=0}^K\frac{1}{t+1} \leq 1+\int_{0}^K\frac{du}{u+1} = 1 + \ln(K+1).\label{appendix_DSLM:Eq9a}
\end{align}
Substituting Eqs.\ \eqref{appendix_DSLM:Eq9} and  \eqref{appendix_DSLM:Eq9a} into Eq.\ \eqref{appendix_DSLM:Eq8} yields 
 \begin{align*}
&\Eset\Big[\big(\bar{\lambda}(K+1) - \lambda^*\big)^2\,|\,\Fcal_k\Big]\notag\\
&\leq\Big(\bar{\lambda}(0)-\lambda^*\Big)^2 + \frac{D^2(1+\ln(K+1))}{n}+ \frac{4D\|\bld(0)\|}{n(1-\delta)}\notag\\
&\quad+ \frac{4D^2(1+\ln(K+1))}{n(1-\delta)}- \frac{2}{n}\sum_{k=0}^{K}\alpha(k)\Big( q(\bar{\lambda}(k)) - q^*\Big),
\end{align*}
which when taking expectation of both sides implies
\begin{align*}
&\Eset\Big[\big(\bar{\lambda}(K+1) - \lambda^*\big)^2\Big]\notag\\
&\qquad \leq\Eset\Big[\Big(\bar{\lambda}(0)-\lambda^*\Big)^2\Big] + \frac{5D^2(1+\ln(K+1))}{n(1-\delta)}\notag\\
&\qquad\qquad + \frac{4D\Eset\Big[\|\bld(0)\|\Big]}{n(1-\delta)} - \frac{2}{n}\sum_{k=0}^{K}\alpha(k)\Eset\Big[ q(\bar{\lambda}(k)) - q^*\Big].
\end{align*} 
Dividing both sides of the preceding relation by $2/n\sum_{k=0}^{K}\alpha(k)$ and rearranging the terms we have
\begin{align}
&\sum_{k=0}^K \frac{\alpha(k) \mathbb{E}\Big[q(\bar{\lambda}(k))\Big]}{\sum_{k=0}^K\alpha(k)}-q^*\nonumber\\
&\qquad \leq \frac{\Eset\Big[n\Big(\bar{\lambda}(0)-\lambda^*\Big)^2\Big]}{2(1-\delta)\sum_{k=0}^K\alpha(k)} + \frac{5D^2(1+\ln(K+1))}{2(1-\delta)\sum_{k=0}^K\alpha(k)}\notag\\
&\qquad\qquad + \frac{2D\Eset\Big[\|\bld(0)\|\Big]}{(1-\delta)\sum_{k=0}^K\alpha(k)}\cdot
\label{appendix_DSLM:Eq10}
\end{align}
Since $\alpha(k) = 1/\sqrt{k+1}$ we have
\begin{align*}
\sum_{k=0}^{K}\alpha(k)= \sum_{k=0}^{K}\frac{1}{\sqrt{k+1}}\geq \int_{0}^{K+1}\frac{du}{\sqrt{u+1}} \geq \sqrt{K+1},
\end{align*}
which when substituting into Eq.\ \eqref{appendix_DSLM:Eq10}, and applying Jensen's inequality we obtain
\begin{align}
&\Eset\left [ q\left(
\frac{\sum_{k=0}^K\alpha(k)\bar{\lambda}(k)}{\sum_{k=0}^K\alpha(k)}\right)\right] - q^*
\nonumber\\
&\qquad \leq \frac{n\Eset\Big[\Big(\bar{\lambda}(0)-\lambda^*\Big)^2\Big]}{2\sqrt{K+1}} + \frac{5D^2(1+\ln(K+1))}{2(1-\delta)\sqrt{K+1}}\notag\\
&\qquad\qquad + \frac{2D\Eset\Big[\|\bld(0)\|\Big]}{(1-\delta)\sqrt{K+1}}\cdot\label{appendix_DSLM:Eq11}
\end{align}
By \eqref{alg_DSLM:yiUpdate}, it is straightforward to verify that
\begin{align*}
y_i(K+1)
&= \frac{\sum_{k=0}^{K}\alpha(k)\lambda_i(k)}{\sum_{k=0}^{K}\alpha(k)}\quad \text{ for all } k\geq 0.
\end{align*}
Thus by the Lipschitz continuity of $q_i$ we have
\begin{align}
&\mathbb{E}\Big[q(y_i(K+1)) - q\left(\begin{array}{l}
\frac{\sum_{k=0}^K\alpha(k)\bar{\lambda}(k)}{\sum_{k=0}^K\alpha(k)}\end{array}\right)\Big]\nonumber\\
&\leq \Eset\left[\sum_{i=1}^n D_i\left|\frac{\sum_{k=0}^{K}\alpha(k)\lambda_i(k)}{\sum_{k=0}^{K}\alpha(k)}-\frac{\sum_{k=0}^K\alpha(k)\bar{\lambda}(k)}{\sum_{k=0}^K\alpha(k)}\right|\right]\nonumber\\
&\leq \frac{D}{\sqrt{K+1}}\sum_{k=0}^K\alpha(k)\Eset\Big[\,\|\bld(k)-\bar{\lambda}(k)\1\|\,\Big]\nonumber\\
&\overset{\eqref{appendix_DSLM:Eq9}}{\leq} \frac{D\Eset\Big[\|\bld(0)\|\Big]}{(1-\delta)\sqrt{K+1}}+ \frac{D^2(1+\ln(K+1))}{(1-\delta)\sqrt{K+1}}\cdot\label{appendix_DSLM:Eq12}
\end{align}
Thus, by adding Eq.\ \eqref{appendix_DSLM:Eq12} to Eq.\ \eqref{appendix_DSLM:Eq11},  we obtain Eq.\ \eqref{alg_DSLM:functionrate}.

\end{proof}